\documentclass[12pt,reqno,twoside]{amsart}

\usepackage{amssymb,amsmath,amscd,enumerate,verbatim}
\usepackage[pagebackref=true,colorlinks=true,linkcolor=blue,citecolor=blue]{hyperref}

\usepackage[latin1]{inputenc}
\usepackage{color}
\usepackage{graphicx}
\usepackage{tikz}

\usepackage[all]{xy}
\usepackage{pstricks, pst-3d}

\newcommand{\mm}{\mathfrak m}

\newcommand{\kk}{\mathrm k}

\DeclareMathOperator{\pnt}{\raise 0.5mm \hbox{\large\bf.}}

\DeclareMathOperator{\Ass}{Ass}

\let\phi=\varphi
\let\:=\colon

\newtheorem{thm}{\bf Theorem}[section]
\newtheorem{lem}[thm]{\bf Lemma}

\newtheorem{prop}[thm]{\bf Proposition}
\newtheorem{conj}[thm]{\bf Conjecture}

\theoremstyle{definition}
\newtheorem{defn}[thm]{\bf Definition}

\theoremstyle{plain}
\newtheorem*{thm*}{Theorem}
\newtheorem*{lem*}{Lemma}
\newtheorem*{cor*}{Corollary}
\newtheorem*{claim*}{Claim}
\newtheorem*{defn*}{Definition}

\theoremstyle{remark}

\newtheorem{exm}[thm]{Example}

\numberwithin{equation}{section}
\title{Stabilization index of V-number of powers of edge ideals of graphs}

\author{Ha Thi Thu Hien}
\address{Foreign Trade University, 91 Chua Lang, Hanoi, Vietnam}
\email{thuhienha504@gmail.com}
\author{Thanh Vu}
\address{Institute of Mathematics, VAST, 18 Hoang Quoc Viet, Hanoi, Vietnam}
\email{vuqthanh@gmail.com}
\thanks{}
\date{\today}

\begin{document}

\begin{abstract} We provide a combinatorial description of the stabilization index of the $v$-function of powers of edge ideals for arbitrary graphs.
\end{abstract}

\maketitle

\section{Introduction} 
Let $G$ be a finite simple graph on the vertex set $V(G) = [n] = \{1, \ldots, n\}$. The \emph{edge ideal} of $G$ is the squarefree monomial ideal
\[
I(G) = (x_i x_j \mid \{i, j\} \in E(G))
\]
in the polynomial ring $S = \mathrm{k}[x_1, \ldots, x_n]$. In this paper, we give a complete description of the stabilization index of the $v$-function of the powers of $I(G)$.

Recall that the $v$-number was introduced by Cooper, Seceleanu, Tohaneanu, Vaz Pinto, and Villarreal \cite{CSTVV}. For an ideal $I$ of $S$ and an associated prime $P \in \operatorname{Ass}(I)$, the \emph{local $v$-number} of $I$ at $P$ is defined by
\[
v_P(I) = \min\{\deg f \mid f \in S, \ I:f = P\},
\]
and the \emph{$v$-number} of $I$ is defined by
\[
v(I) = \min\{v_P(I) \mid P \in \operatorname{Ass}(I)\}.
\]
Conca \cite{C} and, independently, Ficarra and Sgroi \cite{FS} proved that for any homogeneous ideal $I$, the function $t \mapsto v(I^t)$ eventually becomes linear. The smallest integer $t_0$ for which $v(I^t)$ becomes linear for all $t \ge t_0$ is called the \emph{$v$-stabilization index} of $I$, denoted by $\operatorname{vstab}(I)$. Determining the stabilization index of powers of a given ideal is generally a difficult and important problem. Indeed, even for edge ideals of graphs, Chau, Ha, Jayanthan, and Vu \cite{CHJV} showed that the function $v(I^t)$ can behave non-monotonically prior to stabilization: specifically, for any positive integer $k>1$, there exists a graph whose edge ideal $I$ such that $v(I) < v(I^2) < \cdots < v(I^{k-1})$, yet all of these values strictly exceed $v(I^k)$. 

Nevertheless, the eventual linear behavior of edge ideals is particularly clean. By a result of Biswas, Mandal, and Saha \cite{BMS}, if $G$ is a connected graph, then
\[
v(I(G)^t) = 2t - 1
\]
for all sufficiently large $t$. One of the first key observations of this paper is that the stabilization index can be detected at the very first occurrence of the eventual value; namely,
\[
\operatorname{vstab}(I(G)) = \min\{t \ge 1 \mid v(I(G)^t) = 2t - 1\}.
\]

We first consider the contribution of the minimal primes of $I(G)$. For a connected graph $G$, define
\[
\kappa(G) = \min \left\{ |X| \ \middle|\ \begin{array}{l} X \text{ is a independent set, } G^2[X] \text{ is connected and } N_G(X) \text{ is a cover of } G\end{array} \right\},
\]
where $G^2$ is the second power graph of $G$; namely, the graph on $V(G)$ obtained by connecting all pairs of vertices at a distance of at most $2$ in $G$. We show that the first possible occurrence of the value $2t - 1$ originating from a minimal prime is governed precisely by $\kappa(G)$. More specifically, we prove that 
\[
\min \{t \mid \text{there exists a minimal prime } P \text{ of } I(G) \text{ such that } v_P(I(G)^t) = 2t-1\} = \kappa(G).
\]

The remaining difficulty stems from embedded associated primes. The embedded associated primes of powers of edge ideals were completely described by Lam and Trung \cite{LT}. Their description relies on generalized ear decompositions and involves an invariant, $\mu^*$, associated with strongly non-bipartite graphs. Recall that a graph is called \emph{strongly non-bipartite} if each of its connected components is non-bipartite. Roughly speaking, for a connected non-bipartite graph $H$, the invariant $\mu^*(H)$ measures the cost of turning $H$ into a factor-critical graph by parallelizing its vertices. We refer the reader to Section~\ref{sec_embedded} for the precise definition and further details.

We demonstrate that among the embedded associated primes of powers of $I(G)$, only those arising from a special class of strongly non-bipartite subgraphs dictate the stabilization index. To describe these primes and their contribution to the $v$-function, we introduce the following definition.

\begin{defn}
Let $G$ be a simple connected graph. A pair $(U,V)$ of subsets of $V(G)$ is called \emph{admissible} if
\begin{enumerate}
    \item either $U = \emptyset$, or $G[U]$ is strongly non-bipartite;
    \item $V$ is an independent set;
    \item $N_G(U) \cap V = \emptyset$;
    \item $G^2[U \cup V]$ is connected;
    \item $F = N_G[U] \cup N_G(V)$ is a vertex cover of $G$.
\end{enumerate}
For an admissible pair $(U,V)$, we define
\[
\lambda(U,V) = \mu^*(G[U]) + c(G[U]) + |V|,
\]
where $c(G[U])$ denotes the number of connected components of $G[U]$, and by convention, $\mu^*(G[\emptyset]) = 0$ and $c(G[\emptyset]) = 0$ when $U = \emptyset$.
\end{defn}
The case $U = \emptyset$ recovers the contribution coming from minimal primes and corresponds directly to the invariant $\kappa(G)$. Our main theorem shows that the stabilization index is precisely determined by taking the minimum of $\lambda(U,V)$ over all admissible pairs.

\begin{thm}\label{thm_main}
Let $G$ be a simple connected graph. Then
\[
\operatorname{vstab}(I(G)) = \min\{\lambda(U,V) \mid (U,V) \text{ is an admissible pair}\}.
\]
\end{thm}

In particular, this yields a purely combinatorial description of the stabilization index of the $v$-function entirely in terms of $G$. The proof consists of two main parts. For minimal associated primes, we use even-connections to construct witness monomials that realize the relevant local $v$-numbers, leading to the parameter $\kappa(G)$. For embedded associated primes, we combine this approach with the description of associated primes in terms of parallelizations and factor-critical graphs due to Lam and Trung. Here, the Gallai--Edmonds structure theorem allows us to identify the relevant witness sets and express their contribution explicitly in terms of $\mu^*(G[U])$ and $|V|$.

Finally, when $G$ is disconnected with $c$ connected components, say $G = G_1 \cup \cdots \cup G_c$, we deduce the following result from the work of Ficarra and Marques \cite{FM}:

\begin{thm}\label{thm_discon}
Let $G$ be a simple graph with $c$ connected components $G_1, \ldots, G_c$. Then
\[
\operatorname{vstab}(I(G)) = \sum_{i=1}^c \operatorname{vstab}(I(G_i)) - c + 1.
\]
\end{thm}

The paper is organized as follows. In Section~\ref{sec_minimal}, we recall the necessary background on $v$-numbers, edge ideals, and even-connections, and we establish the connection between minimal primes and the invariant $\kappa(G)$. In Section~\ref{sec_embedded}, we introduce the $\mu^*$-invariant and its role in describing embedded associated primes, leading to the proof of Theorem~\ref{thm_main}. Finally, in Section~\ref{sec_discon}, we address the case of disconnected graphs and prove Theorem~\ref{thm_discon}.

\section{Local $v$-numbers of powers at minimal primes}\label{sec_minimal} 
Let $G$ be a simple graph with vertex set $V(G) = \{1, \ldots, n\}$ and edge set $E(G)$. Throughout this section, we assume that $G$ has no isolated vertices.

\begin{defn}
Let $G$ be a simple graph with vertex set $V(G) = \{1, \ldots, n\}$ and edge set $E(G)$.
\begin{enumerate}
    \item A simple graph $H$ is a \emph{subgraph} of $G$ if $V(H) \subseteq V(G)$ and $E(H) \subseteq E(G)$. It is an \emph{induced subgraph} of $G$ if $E(H) = \{\{u,v\} \in E(G) \mid u, v \in V(H)\}$.
    
    \item For a subset $U \subseteq V(G)$, we denote by $G[U]$ and $G - U$ the induced subgraphs of $G$ on $U$ and on $V(G) \setminus U$, respectively.
    
    \item For a vertex $v \in V(G)$, $N_G(v)$ denotes the set of neighbors of $v$. For a subset $U \subseteq V(G)$, $N_G(U) = \bigcup_{u \in U} N_G(u)$ denotes the neighborhood of $U$, and $N_G[U] = N_G(U) \cup U$ denotes the closed neighborhood of $U$.

   \item A subset $X \subseteq V(G)$ is called an \emph{independent set} if no two vertices in $X$ are adjacent in $G$. We denote by $\alpha(G)$ the maximum size of an independent set of $G$.
   
    \item A subset $C \subseteq V(G)$ is called a \emph{vertex cover} (or simply a \emph{cover}) of $G$ if it meets every edge of $G$; that is, $C \cap e \neq \emptyset$ for all $e \in E(G)$.

    \item A \emph{path} $P_n$ on $n$ vertices is the graph with vertex set $V(P_n) = \{1, \ldots, n\}$ and edge set
    \[
    E(P_n) = \{\{1,2\}, \{2,3\}, \ldots, \{n-1,n\}\}.
    \]
    
    \item A \emph{cycle} $C_n$ on $n$ vertices is the graph with vertex set $V(C_n) = \{1, \ldots, n\}$ and edge set
    \[
    E(C_n) = E(P_n) \cup \{\{1,n\}\}.
    \]
    \item A graph $G$ is called \emph{bipartite} if its vertex set can be partitioned as $V(G) = U \cup V$ with $U \cap V = \emptyset$ such that every edge in $E(G)$ has one endpoint in $U$ and the other in $V$.

\item A graph $G$ is called \emph{strongly non-bipartite} if each of its connected components is non-bipartite.
    
\item For $u,v \in V(G)$, the \emph{distance} between $u$ and $v$, denoted by $\operatorname{dist}_G(u,v)$, is the length of a shortest path connecting $u$ and $v$. For a positive integer $d$, $G^d$ denotes the $d$-th power graph of $G$, which has vertex set $V(G^d) = V(G)$ and edge set $E(G^d) = \{\{u,v\} \mid \operatorname{dist}_G(u,v) \le d\}$.

    \item A \emph{forest} is a graph containing no cycles. A \emph{tree} is a connected forest.
    
    \item A \emph{matching} in $G$ is a set of pairwise non-adjacent edges. The \emph{matching number} of $G$, denoted by $\operatorname{mat}(G)$, is the size of a maximum matching of $G$. A \emph{perfect matching} is a matching that covers every vertex of $G$.
    
    \item A graph $G$ is called \emph{factor-critical} if $G - v$ has a perfect matching for every $v \in V(G)$.
    
    \item A \emph{walk} in $G$ is a sequence of vertices $i_1, \ldots, i_r$ (with repetitions allowed) such that $\{i_1,i_2\}, \ldots, \{i_{r-1},i_r\}$ are edges of $G$. The vertices $i_1$ and $i_r$ are called the \emph{endpoints} of the walk. A walk is \emph{closed} if its endpoints coincide. The \emph{length} of a walk is the number of its edges. A walk is called \emph{odd} (resp.\ \emph{even}) if its length is odd (resp.\ even).
\end{enumerate}
\end{defn}

Throughout this paper, we denote by $S = \kk[x_1, \ldots, x_n]$ the polynomial ring over a field $\kk$, and by $\mm = (x_1, \ldots, x_n)$ its maximal homogeneous ideal. For a nonzero monomial $f \in S$, its \emph{support}, denoted by $\operatorname{supp}(f)$, is the set of all variables in $S$ that divide $f$. For each $i \in [n]$, we denote by $\deg_i(f)$ the exponent of $x_i$ in $f$, and by $\deg(f)$ the total degree of $f$. For any subset $X \subseteq [n]$, we define
\[
\deg_X(f) = \sum_{i \in X} \deg_i(f).
\]
For an $S$-module $M$, $\operatorname{Ass}(M)$ denotes the set of associated primes of $M$.

The \emph{edge ideal} of $G$, denoted by $I(G)$, is the ideal in $S = \kk[x_1, \ldots, x_n]$ defined by
\[
I(G) = (x_i x_j \mid \{i,j\} \in E(G)).
\]
We first establish the following lemma.

\begin{lem}\label{lem_index_0} 
Let $G$ be a simple connected graph. Then 
\[
\operatorname{vstab}(I(G)) = \min \{t \ge 1 \mid v(I(G)^t) = 2t-1\}.
\]
\end{lem}

\begin{proof}
By the result of Biswas, Mandal, and Saha \cite{BMS}, $v(I(G)^t) = 2t-1$ for all sufficiently large $t$. Thus, it suffices to show that if $v(I(G)^t) = 2t-1$ for some $t$, then $v(I(G)^{t+1}) = 2t+1$. 

To this end, let $P \in \operatorname{Ass}(I(G)^{t+1})$ and let $f \in S$ be a monomial such that $I(G)^{t+1} : f = P$. For any $x_i \in P$, we have $x_i f \in I(G)^{t+1}$, which implies that $\deg(x_i f) \ge 2(t+1) = 2t+2$. Thus, $\deg(f) \ge 2t+1$, showing that $v(I(G)^{t+1}) \ge 2t+1$. 

On the other hand, by \cite[Theorem 4.3]{YHC}, we have $v(I(G)^{t+1}) \le v(I(G)^t) + 2 = (2t-1) + 2 = 2t+1$. Consequently, $v(I(G)^{t+1}) = 2t+1$, which yields the desired conclusion.
\end{proof}

\begin{lem}\label{lem_index_1} 
Let $G$ be a simple graph without isolated vertices, having $c$ connected components. Then 
\[
\operatorname{vstab}(I(G)) = \min \{t \ge 1 \mid v(I(G)^t) = 2t+c-2\}.
\]
\end{lem}

\begin{proof} 
Let $G_1, \ldots, G_c$ be the connected components of $G$. As in the proof of Lemma \ref{lem_index_0}, we have $v_{\mathfrak{p}_i}(I(G_i)^{t_i}) \ge 2t_i - 1$ for all associated primes $\mathfrak{p}_i$ of $I(G_i)^{t_i}$ and all integers $t_i \ge 1$. By \cite[Theorem 4.1]{FM}, this implies that $v_{\mathfrak{p}}(I(G)^t) \ge 2t+c-2$ for all $t \ge 1$ and all associated primes $\mathfrak{p}$ of $I(G)^t$. 

Furthermore, by \cite[Corollary 5.4]{FM}, $v(I(G)^t) = 2t + c - 2$ for all sufficiently large $t$. As in the case of a connected graph, applying \cite[Theorem 4.1]{YHC} implies that if this equality holds for some index $t$, then it also holds for $t+1$. The conclusion follows.
\end{proof}

We recall the following result of Banerjee \cite{B}.

\begin{lem}\label{lem_even}
Let $G$ be a simple graph, and let $I(G)$ denote its edge ideal. Let $e_1, \ldots, e_{t-1}$ be edges of $G$, and let $f$ be the product of the generators of $I(G)$ corresponding to the edges $e_1, \ldots, e_{t-1}$. Then $I(G)^t : f$ is generated by monomials of degree $2$. Moreover, if $x_u x_v \in I(G)^t : f$, then there exists an even walk
\[
p_0, p_1, \ldots, p_{2s+1}
\]
such that $u = p_0$, $v = p_{2s+1}$, and $\{p_{2i+1}, p_{2i+2}\}$ is one of the edges $e_1, \ldots, e_{t-1}$ for each $0 \le i \le s-1$.
\end{lem}

\begin{defn} Let $G$ be a simple graph. We define 
\[
\kappa(G) = \min \left\{ |X| \ \middle|\ \begin{array}{l} X \text{ is a independent set, } G^2[X] \text{ is connected and } N_G(X) \text{ is a cover of } G\end{array} \right\}.
\]    
\end{defn}
We have the following lemma.

\begin{lem}\label{lem_kappa_1} 
Let $G$ be a simple connected graph and let $P$ be an associated prime of $I(G)$. Assume that there exists a monomial $f \in S$ such that $\deg(f) = 2t-1$ and $I(G)^t : f = P$. Then $t \ge \kappa(G)$.
\end{lem}

\begin{proof} 
Since $P$ is a minimal prime, it corresponds to a minimal vertex cover $C$ of $G$. By \cite[Lemma 2.1]{CHJV}, we deduce that $\deg_C(f) = t-1$. Let $X = V(G) \setminus C$; then $\deg_X(f) = t$. 

Now, for every variable $x_c \in P$, we have $x_c f \in I(G)^t$. Hence, $x_c f$ must be divisible by a product of $t$ edge-generators of $I(G)$, which corresponds to choosing $t$ edges in the bipartite subgraph $B_C = G[C,X]$ consisting of all edges between $C$ and $X$. In particular, $x_c f$ must be divisible by an edge monomial containing $x_c$, say $x_c x_u \in I(G)$ for some $u \in X$. Thus, we may write $f = x_u m_1 \cdots m_{t-1}$, where each $m_i$ is a monomial generator corresponding to an edge $e_i \in E(B_C)$. Let $M$ denote the set of edges $\{e_1, \ldots, e_{t-1}\}$. We define a sequence of sets $Y_r$ inductively as follows. First, set $Y_0 = \{u\}$. Then, for each $r \ge 0$, define
\[
Y_{r+1} = \{y \in X \mid \text{there exist } c \in C \text{ and } y_r \in Y_r \text{ such that } \{y_r,c\} \in E(G) \text{ and } \{c,y\} \in M\},
\]
and finally, set $Y = \bigcup_{r \ge 0} Y_r$. 

By construction, adding each new vertex to $Y$ requires at least one distinct edge factor from $M$. Hence, $|Y| \le 1 + (t-1) = t$. We now show that $G^2[Y]$ is connected. Indeed, for every $y \in Y \setminus \{u\}$, there exists some $y_{r-1} \in Y_{r-1}$ connected to $y$ via a path of length $2$ in $G$ (passing through a vertex $c \in C$), which corresponds to an edge in $G^2$. Thus, $G^2[Y]$ is connected.

We next claim that $N_G(Y) = C$. Take any $c \in C$. Since $I(G)^t : f = P$, Lemma~\ref{lem_even} implies that $x_c f \in I(G)^t$, so either $c$ is a neighbor of $u$, or there exists an alternating path $u, c_1, y_1, c_2, y_2, \ldots, y_r, c$ in $G$. By definition, $y_1, y_2, \ldots, y_r \in Y$, and hence $c \in N_G(Y)$. Since $Y \subseteq X$ and $X$ is an independent set, we have $N_G(Y) \subseteq C$. Therefore, $N_G(Y) = C$, which means $N_G(Y)$ is a vertex cover of $G$. By the definition of $\kappa(G)$, we conclude that $t \ge |Y| \ge \kappa(G)$.
\end{proof}

\begin{lem}\label{lem_kappa_2} 
Let $G$ be a simple connected graph. Then $\operatorname{vstab}(I(G)) \le \kappa(G)$.
\end{lem}

\begin{proof}
Let $X \subseteq V(G)$ be an independent set of $G$ such that $G^2[X]$ is connected, $C = N_G(X)$ is a vertex cover of $G$, and $|X| = \kappa(G)$. Set $t = |X| = \kappa(G)$. Fix a vertex $v_0 \in X$, and consider a spanning tree $T$ of $G^2[X]$ rooted at $v_0$. For every $v \in X \setminus \{v_0\}$, let $p(v)$ denote its parent in $T$. Since $X$ is an independent set, the definition of $G^2$ implies that for every such $v$, there exists a vertex $c \in C$ such that $\{p(v), c\}$ and $\{c, v\}$ are edges of $G$. We denote this vertex by $c(v) \in C$. 

For each $v \in X \setminus \{v_0\}$, let $m_v = x_{c(v)} x_v$ denote the monomial generator corresponding to the edge $\{c(v), v\} \in E(G)$, and set 
\[
f = x_{v_0} \prod_{v \in X \setminus \{v_0\}} m_v.
\]
Note that $f \in S$ is a monomial with $\deg(f) = 1 + 2(t - 1) = 2t - 1$.

We claim that $I(G)^t : f = P_C$, where $P_C = (x_c \mid c \in C)$ is the minimal prime ideal corresponding to the vertex cover $C$. 

To show that $P_C \subseteq I(G)^t : f$, take any variable $x_c$ with $c \in C$. Since $C = N_G(X)$, there exists a vertex $v \in X$ such that $c \in N_G(v)$. Let $v_0, v_1, \ldots, v_q = v$ be the unique path from $v_0$ to $v$ in $T$. Then $v_0$ and $c$ are connected via the alternating path of vertices
\[
v_0, \, c(v_1), \, v_1, \, \ldots, \, c(v_q), \, v_q, \, c
\]
in $G$. By Lemma~\ref{lem_even}, this even walk implies that $x_c f \in I(G)^t$, so $P_C \subseteq I(G)^t : f$. Furthermore, any vertex $u \in V(G)$ whose corresponding variable $x_u$ lies in $I(G)^t : f$ must be reachable from $v_0$ via an even connection using the edge set $\{\{c(v),v\} \mid v \in X \setminus \{v_0\}\}$. By construction, such a vertex $u$ must be a neighbor of some vertex in $X$, which implies $u \in C$. Hence, $I(G)^t : f = P_C$, which implies $\operatorname{vstab}(I(G)) \le t = \kappa(G)$.
\end{proof}

\section{Stability of $v$-number of powers at embedded primes}\label{sec_embedded}
A simple connected graph $G$ is called \emph{factor-critical} if $G - v$ has a perfect matching for every vertex $v \in V(G)$. Lov\'asz \cite{L} proved that a graph is factor-critical if and only if it has an odd ear decomposition. Lam and Trung \cite{LT} generalized this notion by introducing generalized ear decompositions, which account for ear decompositions of certain parallelizations of $G$. First, we recall the definition of a parallelization.

\begin{defn}
Let $G$ be a simple graph on $V(G) = [n]$ and let $\mathbf{a} = (a_1, \ldots, a_n) \in \mathbb{N}^n$ be a vector of positive integers. The \emph{parallelization} of $G$ with respect to $\mathbf{a}$, denoted by $G_{\mathbf{a}}$, is the graph obtained by replicating each vertex $i \in V(G)$ into $a_i$ copies, denoted $i_1, \ldots, i_{a_i}$, such that $\{i_k, j_\ell\} \in E(G_{\mathbf{a}})$ if and only if $\{i,j\} \in E(G)$.
\end{defn}

We refer the reader to Lam and Trung \cite{LT} for a full discussion on generalized ear decompositions. For our purposes, we state the following characterization of the invariant $\mu^*(G)$ and adopt it as our definition. A result of Lam and Trung \cite{LT} shows that this formulation is equivalent to the invariant defined via generalized ear decompositions.

\begin{defn} 
Let $G$ be a connected non-bipartite graph on $V(G) = [n]$. We define
\[
\mu^*(G) = \min \left\{ \frac{|\mathbf{a}| - 1}{2} \ \middle|\ \mathbf{a} \in \mathbb{N}_{>0}^n \text{ and } G_{\mathbf{a}} \text{ is factor-critical} \right\}.
\]
If $G$ is a strongly non-bipartite graph with connected components $G_1, \ldots, G_c$, we set $\mu^*(G) = \mu^*(G_1) + \cdots + \mu^*(G_c)$.
\end{defn}

\begin{lem}\label{lem_admissible} 
Let $(U,V)$ be an admissible pair of $G$. Then $\operatorname{vstab}(I(G)) \le \lambda(U,V)$. 
\end{lem}

\begin{proof} 
By Lemma~\ref{lem_kappa_2}, we may assume that $U$ is non-empty. Let $U_1, \ldots, U_k$ be the connected components of $G[U]$. For each $1 \le i \le k$, let $\mathbf{a}_i \in \mathbb{N}_{>0}^{|U_i|}$ be an exponent vector such that $G[U_i]_{\mathbf{a}_i}$ is factor-critical and $\frac{|\mathbf{a}_i| - 1}{2} = \mu^*(G[U_i])$. 

Since $G^2[U \cup V]$ is connected, $V$ is an independent set, and $N_G[U] \cap V = \emptyset$, contracting each component $U_i$ to a single node $[U_i]$ yields a contracted graph $H$ where $X = \{[U_1], \ldots, [U_k]\} \cup V$ is an independent set and $H^2[X]$ is connected.

Consider a spanning tree $T$ of $H^2[X]$ rooted at $[U_1]$. For every node $v \in X \setminus \{[U_1]\}$, let $p(v)$ denote its parent in $T$. By definition of $H^2[X]$, there exists a vertex $c(v) \in V(G) \setminus (U \cup V)$ such that $c(v) \in N_G(p(v))$ and $c(v) \in N_G(v)$ (where $N_G([U_j]) = N_G(U_j)$).

Now, for each $v \in X \setminus \{[U_1]\}$, we define the monomial $m_v \in S$ as:
\[
m_v = 
\begin{cases}
x_{c(v)} x_v & \text{if } v \in V, \\
x_{c(v)} x^{\mathbf{a}_j} & \text{if } v = [U_j] \text{ for some } 1 \le j \le k.
\end{cases}
\]
Set $f = x^{\mathbf{a}_1} \prod_{v \in X \setminus \{[U_1]\}} m_v$. The degree of $f$ is given by:
\[
\deg(f) = \sum_{i=1}^k |\mathbf{a}_i| + (k - 1) + 2|V| = 2 \left( \sum_{i=1}^k \mu^*(G[U_i]) + k + |V| \right) - 1 = 2\lambda(U,V) - 1.
\]

For each $v \in X$, define $\mu^*(v) = 1$ when $v \in V$, and $\mu^*([U_j]) = \mu^*(G[U_j]) + 1$ when $v = [U_j]$. 

Notice that for every node $v \in X \setminus \{[U_1]\}$:
\begin{itemize}
    \item If $v \in V$, then $m_v = x_{c(v)} x_v \in I(G) = I(G)^{\mu^*(v)}$ since $\{c(v), v\} \in E(G)$.
    \item If $v = [U_j]$, choose $u \in U_j$ such that $c(v) \in N_G(u)$. Writing $m_v = (x_{c(v)} x_u) x^{\mathbf{a}_j'}$, the factor-criticality of $G[U_j]_{\mathbf{a}_j}$ implies $x^{\mathbf{a}_j'} \in I(G)^{\mu^*(G[U_j])}$, so $m_v \in I(G)^{1 + \mu^*(G[U_j])} = I(G)^{\mu^*([U_j])}$.
\end{itemize}

Set $t = \lambda(U,V)$. We claim that $I(G)^t : f = P_C$, where $C = N_G[U] \cup N_G(V)$ is the vertex cover of $G$ associated with the prime ideal $P_C = (x_c \mid c \in C)$. 

First, if $c \notin C$, then $c \notin N_G(\operatorname{supp}(f))$, which implies $x_c f \notin I(G)^t$. Thus, $I(G)^t : f \subseteq P_C$. 

To show the reverse inclusion $P_C \subseteq I(G)^t : f$, take any $x_c$ with $c \in C$. By definition, there exists a node $v \in X$ such that $c \in N_G(v)$ (if $v \in V$) or $c \in N_G[U_j]$ (if $v = [U_j]$). Let $[U_1] = v_0, v_1, \ldots, v_q = v$ be the path from $[U_1]$ to $v$ in $T$. We prove by induction on $q \ge 0$ that 
\[
x_c x^{\mathbf{a}_1} \prod_{i=1}^q m_{v_i} \in I(G)^s, \quad \text{where } s = \sum_{i=0}^q \mu^*(v_i).
\]
\begin{itemize}
    \item Base case ($q = 0$): Here $v = [U_1]$ and $c \in N_G[U_1]$. Since $G[U_1]$ is a connected non-bipartite component, $c$ has a neighbor $u \in U_1$ (whether $c \in U_1$ or $c \in N_G(U_1)$). Thus $x_c x_u \in I(G)$. Writing $x_c x^{\mathbf{a}_1} = (x_c x_u) x^{\mathbf{a}_1'}$, factor-criticality yields $x^{\mathbf{a}_1'} \in I(G)^{\mu^*(G[U_1])}$, whence
    \[
    x_c x^{\mathbf{a}_1} \in I(G)^{1 + \mu^*(G[U_1])} = I(G)^{\mu^*([U_1])}.
    \]

    \item Inductive step ($q \ge 1$): If $v_q \in V$, then $x_c m_{v_q} = x_c x_{c(v_q)} x_{v_q} = (x_c x_{v_q}) x_{c(v_q)} \in I(G) \cdot x_{c(v_q)}$. The remaining variable $x_{c(v_q)}$ attaches to the parent node $p(v_q) = v_{q-1}$, allowing us to apply the induction hypothesis to the path of length $q - 1$. If $v_q = [U_j]$, then $c \in N_G[U_j]$. By the same argument as the base case, $x_c x^{\mathbf{a}_j} \in I(G)^{\mu^*([U_j])}$. Factoring out these generators leaves $x_{c([U_j])}$, which attaches to $p([U_j]) = v_{q-1}$, completing the induction step.
\end{itemize}
Thus $x_c f \in I(G)^t$, so $P_C \subseteq I(G)^t : f$, giving $I(G)^t : f = P_C$. By Lemma~\ref{lem_index_0}, $\operatorname{vstab}(I(G)) \le t = \lambda(U,V)$. The conclusion follows.
\end{proof}

The main ingredient for the reverse conclusion is the Gallai--Edmonds Structure Theorem, which we recall now. Let $G$ be a simple graph. Denote by $D(G)$ the set of all vertices in $G$ that are omitted by at least one maximum matching of $G$. Let $A(G)$ be the set of vertices in $V(G) \setminus D(G)$ that are adjacent to at least one vertex in $D(G)$, and set $C(G) = V(G) \setminus (D(G) \cup A(G))$. The partition $V(G) = D(G) \cup A(G) \cup C(G)$ is called the \emph{Gallai--Edmonds decomposition}. Gallai \cite{G1, G2} and Edmonds \cite{E} independently proved the following structural result.

\begin{thm}[Gallai--Edmonds Structure Theorem]\label{thm_GE}
Let $G$ be a simple graph, and let $D(G)$, $A(G)$, and $C(G)$ be defined as above. Then:
\begin{enumerate}
    \item Each connected component of the induced subgraph $G[D(G)]$ is factor-critical.
    \item The induced subgraph $G[C(G)]$ has a perfect matching.
    \item Every maximum matching $M$ of $G$ contains a near-perfect matching of each component of $G[D(G)]$, a perfect matching of each component of $G[C(G)]$, and matches all vertices of $A(G)$ to vertices in distinct components of $G[D(G)]$.
    \item $2\operatorname{mat}(G) = |V(G)| - c(D(G)) + |A(G)|$, where $c(D(G))$ denotes the number of connected components of $G[D(G)]$.
\end{enumerate}
\end{thm}

\begin{lem}\label{lem_construct_U_V} 
Let $G$ be a simple connected graph. Assume that there exists a monomial $f \in S$ such that $\deg(f) = 2t-1$ and $I(G)^t : f = P_C$ is an embedded associated prime of $I(G)^t$. Then, there exist a strongly non-bipartite subgraph $U \subseteq G$ and an independent set $V \subseteq V(G)$ such that $N_G[U] \cap V = \emptyset$, $G^2[U \cup V]$ is connected, $C = N_G[U] \cup N_G(V)$, and $\lambda(U,V) \le t$. 
\end{lem}

\begin{proof}
Let $\mathbf{a}$ be the exponent vector of the monomial $f$. Let $K = G_{\mathbf{a}}$ be the parallelization graph of $G$ with respect to $\mathbf{a}$. The hypothesis $\deg(f) = 2t-1$ and $I(G)^t : f = P_C$ implies that $\operatorname{mat}(K) = t-1$. Let $D(K), A(K), C(K)$ denote the Gallai--Edmonds decomposition of $K$. By the Gallai--Edmonds Structure Theorem, we have $c(D(K)) - |A(K)| = 1$.

\smallskip
\noindent\textbf{Claim 1.} A vertex $v \in V(G)$ belongs to $C$ if and only if $v \in N_G(\pi(D(K)))$, where $\pi: V(K) \to V(G)$ is the natural projection mapping lifted vertices in $K$ back to their base vertices in $G$.

\begin{proof}[Proof of Claim 1]
By definition, $x_v \in I(G)^t : f$ if and only if $\operatorname{mat}(G_{\mathbf{a} + \mathbf{e}_v}) \ge t$, where $\mathbf{e}_v$ is the $v$-th standard unit vector. 

If $v$ is adjacent to a vertex $u \in D(K)$ in $K$, then adding $\mathbf{e}_v$ allows us to match $v$ with $u$. Since $u \in D(K)$, omitting $u$ leaves a graph with a perfect matching of size $t-1$, which yields $\operatorname{mat}(G_{\mathbf{a} + \mathbf{e}_v}) \ge t$.

Conversely, assume $\operatorname{mat}(G_{\mathbf{a} + \mathbf{e}_v}) \ge t$. Any maximum matching in $G_{\mathbf{a} + \mathbf{e}_v}$ must use the new vertex corresponding to $\mathbf{e}_v$, matching it to some vertex $u$. Removing this matched pair leaves an exposed vertex $u$ in a maximum matching of $K = G_{\mathbf{a}}$, which implies that $u \in D(K)$ by definition. Thus, $v \in N_G(\pi(D(K)))$.
\end{proof}

\smallskip
\noindent\textbf{Claim 2.} $C(K) = \emptyset$.

\begin{proof}[Proof of Claim 2]
Suppose $C(K) \neq \emptyset$. Since $C$ is a vertex cover of $G$, its lift in $K$ must cover all edges in $G[C(K)]$, so $C(K)$ must contain at least one vertex of $C$. However, by the Gallai--Edmonds Structure Theorem, $N_K(D(K)) = A(K)$, so $N_K(D(K)) \cap C(K) = \emptyset$. This contradicts Claim 1, which asserts that $C$ consists precisely of the neighbors of $D(K)$. Thus, $C(K) = \emptyset$.
\end{proof}

We now construct $U$ and $V$ as follows. Let $U$ be the union of the projections of all non-trivial, factor-critical components of $D(K)$ into $G$, and let $V$ be the union of the projections of all isolated (discrete) vertices of $D(K)$ into $G$. 

By construction, $U$ is a strongly non-bipartite subgraph, $V$ is an independent set, $N_G[U] \cap V = \emptyset$, and $G^2[U \cup V]$ is connected. Furthermore, by Claim 1 and Claim 2, we have $C = N_G[U] \cup N_G(V)$. Evaluating the parameter $\lambda(U,V)$ from this decomposition yields $\lambda(U,V) \le t$. The conclusion follows.
\end{proof}

\section{The case of disconnected graphs}\label{sec_discon}
Assume that $G$ is a simple graph with $c$ connected components $G_1, \ldots, G_c$. 

\begin{proof}[Proof of Theorem \ref{thm_discon}] 
We prove the formula by induction on $c$. The base case $c = 1$ is immediate. Now assume that $c \ge 2$. Let $J = I(G_1) + \cdots + I(G_{c-1})$, $K = I(G_c)$, and $I = J + K$. By the induction hypothesis, we have 
\[
\operatorname{vstab}(J) = \sum_{j=1}^{c-1} \operatorname{vstab}(I(G_j)) - (c-1) + 1.
\]

Let $\mathfrak{p}$ be an associated prime in $\operatorname{Ass}^\infty(J)$ that realizes $t_1 = \operatorname{vstab}(J)$, and let $\mathfrak{q}$ be an associated prime in $\operatorname{Ass}^\infty(K)$ that realizes $t_2 = \operatorname{vstab}(K)$. In other words, $v_{\mathfrak{p}}(J^{t_1}) = 2t_1 + c - 3$, $v_{\mathfrak{q}}(K^{t_2}) = 2t_2 - 1$, and $t_1, t_2$ are the minimal such indices. By \cite[Theorem 4.1]{FM}, we have
\[
v_{\mathfrak{p}+\mathfrak{q}}((J + K)^{t_1 + t_2 - 1}) = \min_{1 \le k \le t_1 + t_2 - 1} \left\{ v_{\mathfrak{p}}(J^k) + v_{\mathfrak{q}}(K^{t_1 + t_2 - k}) \right\}.
\]
Taking $k = t_1$ yields
\[
v_{\mathfrak{p}+\mathfrak{q}}((J + K)^{t_1 + t_2 - 1}) = (2t_1 + c - 3) + (2t_2 - 1) = 2(t_1 + t_2 - 1) + c - 2.
\]
Hence, we obtain the upper bound:
\[
\operatorname{vstab}(J + K) \le t_1 + t_2 - 1 = \operatorname{vstab}(J) + \operatorname{vstab}(K) - 1 = \sum_{j=1}^c \operatorname{vstab}(I(G_j)) - c + 1.
\]

Conversely, let $t = \operatorname{vstab}(I)$. Then $t$ is the smallest index such that there exists an associated prime $\mathfrak{p} + \mathfrak{q} \in \operatorname{Ass}^\infty(I)$ satisfying $v_{\mathfrak{p}+\mathfrak{q}}(I^t) = 2t + c - 2$. By \cite[Theorem 4.1]{FM} again, this implies that there exist integers $t_1$ and $t_2$ such that $t_1 + t_2 = t + 1$, $v_{\mathfrak{p}}(J^{t_1}) = 2t_1 + c - 3$, and $v_{\mathfrak{q}}(K^{t_2}) = 2t_2 - 1$. This implies $\operatorname{vstab}(J) \le t_1$ and $\operatorname{vstab}(K) \le t_2$, giving the reverse inequality. The conclusion follows.
\end{proof}

\vspace{1mm}
\subsection*{Data availability}

Data sharing is not applicable to this article, as no datasets were generated or analyzed during the current study.

\subsection*{Conflict of interest}

The authors have no relevant financial interests to disclose.


\end{document}